\documentclass{amsart}

\newtheorem{theorem}{Theorem}[section]
\newtheorem{lemma}[theorem]{Lemma}

\newtheorem{problem}[theorem]{Problem}

\theoremstyle{definition}

\usepackage{amsmath,amssymb,amsfonts}

\theoremstyle{remark}

\numberwithin{equation}{section}

\newcommand{\R}{\mathbb{R}}
\newcommand{\N}{\mathbb{N}}

\newcommand{\X}{\mathrm{X}}

\newcommand{\Y}{\mathrm{Y}}
\newcommand{\Z}{\mathrm{Z}}
\newcommand{\B}{\mathbf{B}}
\newcommand{\I}{\mathbf{I}}

\renewcommand{\S}{\mathbf{S}}

\begin{document}

\title[Operators on $C_{0}(L,\X)$ whose range does not contain $c_{0}$]{Operators on $C_{0}(L,\X)$ whose range\\ does not contain $c_{0}$}

\author{Jarno Talponen}
\address{University of Helsinki, Department of Mathematics and Statistics, Box 68, \\ 
(Gustaf H\"{a}llstr\"{o}minkatu 2b) FI-00014 University of Helsinki, Finland}
\email{talponen@cc.helsinki.fi}

\date{\today}

\begin{abstract}
This paper contains the following results:\\ 
(a) Suppose that $\X\neq \{0\}$ is a Banach space and $(L,\tau)$ is a non-empty locally compact Hausdorff space without isolated points.  
Then each linear operator $T\colon C_{0}(L,\X)\rightarrow C_{0}(L,\X)$ whose range does not contain an isomorphic copy
of $c_{00}$ satisfies the Daugavet equality $||\I+T||=1+||T||$.\\
(b) Let $\Gamma$ be a non-empty set and $\X,\Y$ be Banach spaces such that $\X$ is reflexive and $\Y$ does not contain 
$c_{0}$ isomorphically. Then any continuous linear operator $T\colon c_{0}(\Gamma,\X)\rightarrow \Y$ is weakly compact.
\end{abstract}
\maketitle

\section{Introduction}
In what follows $\X,\Y$ and $\Z$ are real Banach spaces different from $\{0\}$, 
$L$ is a non-empty locally compact Hausdorff space and $C_{0}(L,\X)$ denotes the $||\cdot||_{\infty}$-normed Banach space of 
$\X$-valued continuous functions on $L$ vanishing at infinity. This work is related to the result due to P. Cembranos \cite{Cembranos1} 
that if $K$ is an infinite compact and $\X$ is an infinite-dimensional Banach space, 
then the Banach space $C(K,\X)$ of $\X$-valued continuous functions on $K$ contains a complemented copy of $c_{0}$. 
See also the related paper \cite{Cembranos2} where the Dieudonn\'{e} property of $C(K,\X)$ is studied.
  
Taking another direction, we will study continuous linear operators of the type\\ 
($\ast$)\ $T\colon C_{0}(L,\X)\rightarrow \Y$, where $\Y$ does not contain $c_{0}$ isomorphically.\\ 
Above there is a structural disparity between spaces $C_{0}(L,\X)$ and $\Y$, 
since typically the former space contains copies of $c_{0}$ in abundance. This difference has a strong impact on the properties of $T$. 
Namely, it turns out that the range of $T$ in $\Y$ is \emph{small} in some sense. 

If $L$ does not contain isolated points, then an operator $T\colon C_{0}(L,\X)\rightarrow C_{0}(L,\X)$ of type ($\ast$) satisfies the 
Daugavet type equality $||\I+T||=1+||T||$ (see Theorem \ref{thm1}). See \cite{Werner} for recent discussion on matters related to the 
Daugavet property.

If $L$ is discrete, $\X$ is reflexive and $c_{0}$ is not contained in $\Y$ isomorphically, then an operator 
$T\colon C_{0}(L,\X)\rightarrow \Y$ is weakly compact (see Theorem \ref{thm2}).  

\subsection*{Preliminaries}
Here $\X$ and $\Y$ denote real Banach spaces. The closed unit ball and the unit sphere of $\X$ are denoted by $\B_{\X}$ and $\S_{\X}$,
respectively. An identity mapping is denoted by $\I$. An operator $T\colon\X\rightarrow\Y$ is weakly compact if 
$\overline{T(\B_{\X})}$ is weakly compact. If $\X\neq \{0\}$, then we say that $\X$ is \emph{non-trivial}.
For given sets $A\subset B$ the mapping $\chi_{A}\colon B\rightarrow \{0,1\}$ is determined by $\chi_{A}(t)=1$ if and only if $t\in A$.
We refer to \cite{Diestel}, \cite{HHZ} and \cite{Willard} for suitable background information including definitions and basic results.

\section{Results}

\begin{theorem}\label{thm1}
Let $\X$ be a non-trivial Banach space and $(L,\tau)$ a non-empty locally compact Hausdorff space without isolated points.  
Then each linear operator $T\colon C_{0}(L,\X)\rightarrow C_{0}(L,\X)$ whose range does not contain an isomorphic copy
of $c_{00}$ satisfies the Daugavet equality 
\[||\I+T||=1+||T||.\] 
\end{theorem}

Let us first make some preparations before giving the proof. It is easy to see that the range of $T$ contains $c_{00}$ isomorphically
if and only if the closure of the range contains $c_{0}$.

The assumption that $L$ does not contain isolated points cannot be removed. 
Indeed, if $L$ is not a singleton, $t_{0}\in L$ is an isolated point and $\X$ contains no isomorphic copy of $c_{0}$, 
then the linear operator $T\colon C_{0}(L,\X)\rightarrow \X;\ F \mapsto -\chi_{\{t_{0}\}}(\cdot)F(\cdot)$ 
is of type ($\ast$) and satisfies $||T||=||\I+T||=1$.

The above Theorem \ref{thm1} holds analogously for $T\colon CB(L,\X)\rightarrow CB(L,\X)$, essentially with the same proof.
Here $CB(L,\X)$ is the $||\cdot||_{\infty}$-normed Banach space of $\X$-valued bounded continuous functions on $L$.

For a linear operator $T\colon C_{0}(L,\X)\rightarrow \Y$ we denote
\[\mathrm{osc}_{T}(A)=\sup \{||TF||:\ F\in \B_{C_{0}(L,\X)},\ L\setminus U\subset F^{-1}(0)\}\quad \mathrm{for}\ A\subset L.\]

\begin{lemma}\label{lemma}
Let $T\colon C_{0}(L,\X)\rightarrow \Y$ be a linear operator, where $\Y$ does not contain $c_{0}$ isomorphically.
Suppose that $(V_{n})_{n\in\N}$ is a sequence of pair-wise disjoint non-empty open subsets of $L$.
Then $\mathrm{osc}_{T}(V_{n})\rightarrow 0$ as $n\rightarrow 0$.
\end{lemma}
\begin{proof} 
By passing to a subsequence it suffices, without loss of generality, to show that $\inf_{n\in\N}\mathrm{osc}_{T}(V_{n})=0$. 
Indeed, assume to the contrary that there is some $d>0$ such that $\mathrm{osc}_{T}(V_{n})\geq d$ for all $n\in\N$. 
This means that one can find a sequence 
$(F_{n})_{n\in\N}\subset (\frac{1}{d}+1)\B_{C_{0}(L,\X)}$ such that $F_{n}$ is supported in $V_{n}$ and
$||T(F_{n})||=1$ for $n\in\N$. Note that for each finite subset $I\subset \N$ it holds that 
$\sum_{i\in I}F_{i}\in (\frac{1}{d}+1)\B_{C_{0}(L,\X)}$ as $V_{n}$ are pair-wise disjoint. 
Since $T$ is linear and continuous, we obtain that 
\[\sup_{\epsilon,I}\left|\left|\sum_{i\in I} T(\epsilon_{i}F_{i})\right|\right|\leq \left(\frac{1}{d}+1\right)||T||,\]
where the supremum is taken over all signs $\epsilon\colon \N \rightarrow \{-1,1\}$ and finite subsets $I\subset\N$. 

Recall the well-known result due to Bessaga and Pelczynski (see e.g. \cite[p.202]{HHZ}) that in a Banach space $\Y$ 
a sequence $(y_{n})\subset S_{\Y}$ is equivalent to the standard unit vector basis of $c_{0}$ if and only if 
$\sup_{\epsilon,I} ||\sum_{i\in I} \epsilon_{i}y_{i}||<\infty$ (supremum taken as above). 
By placing $y_{i}=T(F_{i})$ we obtain that the range of $T$ contains $c_{00}$ isomorphically, which 
contradicts the assumptions. Hence $\inf_{n\in\N}\mathrm{osc}_{T}(V_{n})=0$.
\end{proof}

\begin{proof}[Proof of Theorem \ref{thm1}]

Recall that as $(L,\tau)$ is a locally compact Hausdorff space it is completely regular, that is, for each closed set 
$C\subset L$ and $t\in L\setminus C$ there is a continuous map $s\colon L\rightarrow \R$ such that
$s(C)=\{0\}$ and $s(t)=1$.
 
Suppose that there are no isolated points in 
$(L,\tau)$. Let $T\colon C_{0}(L,\X)\rightarrow C_{0}(L,\X)$ be a linear operator.
If the operator norm of $T$ is $0$ or $\infty$, then the Daugavet equation holds trivially, 
so that we may concentrate on the case $||T||=C\in (0,\infty)$. Let $k\in \N$. 
Fix $F\in \S_{C_{0}(L,\X)}$ such that $G=TF$ satisfies $||G||> C-\frac{1}{k}$. Consider the 
open subspace $U=\{t\in L:\ ||G(t)||> C-\frac{1}{k}\}$ of $L$. 

We can pick a sequence $(V_{n})_{n\in\N}\subset U$ of pair-wise disjoint open subsets as follows. 
Clearly $\overline{U}$ is also a locally compact (even compact) Hausdorff space which does not contain isolated points. 
Hence $U$ itself is not a singleton, and we may take two points $t_{0},t_{1}\in U ,\ t_{0}\neq t_{1}$. 
Since $U$ is a Hausdorff space, there are disjoint open neighbourhoods $U_{0},U_{1}\subset U$ of $t_{0}$ and $t_{1}$, respectively. 
By repeating the same reasoning, pick $t_{10},t_{11}\in U_{1},\ t_{10}\neq t_{11}$ and disjoint open neigbourhoods 
$U_{10},U_{11}\subset U_{1}$ of $t_{10}$ and $t_{11}$, respectively. Similarly, pick $t_{110},t_{111}\in U_{11},\ t_{110}\neq t_{111}$ 
and the corresponding disjoint open neighbourhoods $U_{110},U_{111}\subset U_{11}$. Proceeding in this manner yields 
a sequence of pair-wise disjoint open subsets by letting $V_{n}=U_{s},\ s\in \{1\}^{n}\times \{0\}$ for $n\in\N$.

Since $c_{0}\not\subset \overline{T(C_{0}(L,\X))}$ we obtain by using Lemma \ref{lemma} that 
$\mathrm{osc}_{T}(V_{n})\rightarrow 0$ as $n\rightarrow\infty$. Fix $n\in\N$ such that $\mathrm{osc}_{T}(V_{n})<\frac{1}{k}$. 
Let $u_{0}\in V_{n}$. By using the complete regularity of $L$ 
one can find a continuous map $s\colon L\rightarrow [0,1]$ such that $s(L\setminus V_{n})=\{0\}$ and $s(u_{0})=1$. 
Observe that the mappings $s(\cdot)F(\cdot)$ and $\frac{s(\cdot)}{\max(1,||G(\cdot)||_{\X})}G(\cdot)$ are elements of $\B_{C_{0}(L,\X)}$.
Hence 
\begin{equation}\label{Pk}
||T(s(\cdot)F(\cdot))||_{C_{0}(L,\X)}\leq \frac{1}{k}\ \mathrm{and}\ \left|\left|T\left(\frac{s(\cdot)}{\max(1,||G(\cdot)||_{\X})}G(\cdot)\right)\right|\right|_{C_{0}(L,\X)}\leq \frac{1}{k}
\end{equation}
by the definition of $\mathrm{osc}_{T}(V_{n})$. Note that 
\[\left|\left|(1-s(t))F(t)-\frac{s(t)}{\max(1,||G(t)||_{\X})}G(t)\right|\right|_{\X}\leq (1-s(t))||F(t)||_{\X}+s(t)\leq 1\]
for all $t\in L$. Hence $E(\cdot)\stackrel{\cdot}{=}(1-s(\cdot))F(\cdot)+\frac{s(\cdot)}{\max(1,||G(\cdot)||_{\X})}G(\cdot)$ 
defines an element of $\B_{C_{0}(L,\X)}$. Note that $E$ is a kind of interpolation of $F$ and $G$.

Observe that 
\[||G-T(E)||\leq \frac{2}{k}\]
according to \eqref{Pk}, and that
\begin{equation*}
\begin{array}{rcl} 
||E+G||_{C_{0}(L,\X)}&\geq &||(E+G)u_{0}||_{\X}=\left|\left|\frac{s(u_{0})}{||G(u_{0})||_{\X}}G(u_{0})+G(u_{0})\right|\right|_{\X}\\
&=&1+||G(u_{0})||_{\X}>1+C-\frac{1}{k}.
\end{array}
\end{equation*}
Thus
\[||\I+T||\geq ||E+T(E)||\geq ||E+G||-||G-T(E)||>1+C-\frac{3}{k}\]
and by letting $k\rightarrow\infty$ we obtain that $||\I+T||\geq 1+C=1+||T||$. By the triangle inequality $||\I+T||\leq 1+||T||$ 
and we have the claim.
\end{proof}

Let us recall a few classical results due to James which are applied here frequently: 
A closed convex subset $C\subset\X$ is weakly compact if and only if each $f\in\X^{\ast}$ attains its supremum over $C$, and
$\X$ is reflexive if and only if $\B_{\X}$ is weakly compact (see e.g. \cite[Ch.3]{HHZ}).

\begin{theorem}\label{thm2}
Let $\Gamma$ be a non-empty set and $\X,\Y$ be Banach spaces such that $\X$ is reflexive and $\Y$ does not contain 
$c_{0}$ isomorphically. Then any continuous linear operator $T\colon c_{0}(\Gamma,\X)\rightarrow \Y$ is weakly compact.
\end{theorem}

The above result holds similarly for $\ell^{\infty}(\Gamma,\X)$ in place of $c_{0}(\Gamma,\X)$, essentially with the same proof.
Note that the operators $\I\colon c_{0}(\N,\X)\rightarrow c_{0}(\N,\X)$ and 
$T\colon c_{0}(\N,\Y)\rightarrow \Y;\ (y_{n})\mapsto (y_{1},0,0,\ldots)$ are not weakly compact
for any non-trivial $\X$ and non-reflexive $\Y$ according to the James characterization of reflexivity. 
Hence neither of the assumptions about the reflexivity or the non-containment of $c_{0}$ can be removed.

\begin{proof}[Proof of Theorem \ref{thm2}]
Let $T\colon c_{0}(\Gamma,\X)\rightarrow \Y$ be a continuous linear operator such that $\overline{T(\Y)}$ does not contain $c_{0}$.
One may write $c_{0}(\Gamma,\X)=C_{0}(\Gamma,\X)$ isometrically where $\Gamma$ on the right hand side is interpreted as a discrete 
topological space. 

We claim that the sum $\sum_{\gamma\in \Gamma}\mathrm{osc}_{T}(\{\gamma\})$ is defined and finite. Indeed, otherwise 
one can extract pair-wise disjoint (open) subsets $\Gamma_{n}\subset \Gamma,\ n\in\N,$ such that 
$\sum_{\gamma\in \Gamma_{n}}\mathrm{osc}_{T}(\{\gamma\})\geq 1$ for $n\in\N$. However, since $c_{0}\not\subset \overline{T(\Y)}$,
Lemma \ref{lemma} yields that this case does not occur.
Thus there exists a sequence $(\gamma_{n})_{n\in\N}\subset \Gamma$ such that 
$\mathrm{osc}_{T}(\Gamma\setminus \{\gamma_{n}\}_{n\leq k})\rightarrow 0$ as $k\rightarrow \infty$.

In order to verify the statement of the theorem we must show that $\overline{T(\B_{c_{0}(\Gamma,\X)})}$ is weakly compact. 
In doing this we will apply the James characterization of weakly compact sets.
Fix $f\in \Y^{\ast}$. It suffices to show that $f$ attains its supremum over $\overline{T(\B_{c_{0}(\Gamma,\X)})}$.
Observe that $f\circ T$ defines an element of $C_{0}(\Gamma,\X)^{\ast}$.
For each $k\in \N$ define a contractive linear projection 
$P_{k}\colon c_{0}(\Gamma,\X)\rightarrow c_{0}(\Gamma,\X)$ by $P_{k}f(\cdot)=\chi_{\{\gamma_{n}\}_{n\leq k}}(\cdot)f(\cdot)$.
Put $g_{k}=f\circ T\circ P_{k}$ and $\Z_{k}=P_{k}(c_{0}(\Gamma,\X))$ for $k\in\N$. 
Note that $g_{k}$ restricted to $\Z_{k}$ satisfies ${g_{k}|}_{\Z_{k}}\in \Z_{k}^{\ast}$, 
where $\Z_{k}^{\ast}=\ell^{1}(\{\gamma_{n}\}_{n\leq k},\X^{\ast})$ isometrically for $k\in\N$. 
Clearly ${g_{k+l}|}_{\Z_{k}}={g_{k}|}_{\Z_{k}}$ for $k,l\in \N$. Hence there is a sequence 
$(x_{n}^{\ast})_{n\in\N}\subset \X^{\ast}$  such that
\begin{equation}\label{eq: gzk}
g_{k}\left(\sum_{n=1}^{k}\chi_{\{\gamma_{n}\}}(\cdot)y_{n}\right)=\sum_{n=1}^{k}x_{n}^{\ast}(y_{n})
\end{equation}
for $\sum_{n=1}^{k}\chi_{\{\gamma_{n}\}}y_{n}\in \Z_{k},\ k\in \N$.

Observe that since $\X$ is reflexive, there exists according to the James characterization of reflexivity a sequence 
$(x_{n})_{n\in\N}\subset \S_{\X}$ such that $x_{n}^{\ast}(x_{n})=||x_{n}^{\ast}||$ for $n\in\N$.
It follows that
\begin{equation}\label{eq: together}
\sum_{n=1}^{k}x_{n}^{\ast}(x_{n})=||g_{k}||\quad \mathrm{for}\ k\in\N.
\end{equation}

Now, since $\mathrm{osc}_{T}(\Gamma\setminus \{\gamma_{n}\}_{n\leq k})\rightarrow 0$ as $k\rightarrow \infty$,
we obtain that 
\begin{equation}\label{eq: T}
\lim_{k\rightarrow\infty}||T-T\circ P_{k}||=0\ \mathrm{and}\ \lim_{k\rightarrow\infty}||g_{k}-f\circ T||=0.
\end{equation}
By putting these observations together and using the continuity of $T$ we obtain that the sequence 
$\left(T\left(\sum_{n=1}^{k}\chi_{\{\gamma_{n}\}}(\cdot)x_{n}\right)\right)_{k\in\N}\subset\Y$ is Cauchy.
On the other hand,
\[||f\circ T||\geq f\circ T\left(\sum_{n=1}^{k}\chi_{\{\gamma_{n}\}}(\cdot)x_{n}\right)\geq ||g_{k}||-||g_{k}-f\circ T||\rightarrow ||f\circ T||\quad \mathrm{as}\ k\rightarrow \infty\] 
by using \eqref{eq: gzk}, \eqref{eq: together} and \eqref{eq: T}.
We conclude that 
\[y\stackrel{\cdot}{=}\lim_{k\rightarrow \infty}T\left(\sum_{n=1}^{k}\chi_{\{\gamma_{n}\}}(\cdot)x_{n}\right)\in \overline{T(\B_{c_{0}(\Gamma,\X)})}\] 
satisfies $f(y)=||f\circ T||=\sup_{z\in T(B_{c_{0}(\Gamma,\X)})}f(z)$, which completes the proof.
\end{proof}

\begin{problem}
We do not know if a linear operator $T\colon C_{0}(L,\X)\rightarrow \Y$ is weakly compact if $c_{0}\not\subset \Y$
and $L$ is $0$-dimensional.
\end{problem}


\begin{thebibliography}{10}
\bibitem{Cembranos2}
F. Bombal, P. Cembranos, The Dieudonn\'{e} Property of $C(K,E)$, Trans. Amer. Math. Soc. Vol 285, (1984), p. 649-656. 
\bibitem{Cembranos1}
P. Cembranos, $C(K,E)$ contains a complemented copy of $c_{0}$, Proc. Amer. Math. Soc., Vol 91, (1984) p. 556-558.
\bibitem{Diestel}
J. Diestel, J. Uhl Jr., \emph{Vector measures}, Mathematical Surveys, No. 15. American Mathematical Society, Providence, R.I., 1977.
\bibitem{HHZ}
P. Habala, P. Hajek, V. Zizler, \emph{Introduction to Banach Spaces I-II}, matfyzpress, 1996. 
\bibitem{Werner}
D. Werner, Recent progress on the Daugavet property, Irish Math. Soc. Bull.  No. 46  (2001), p. 77-97. 
\bibitem{Willard}
S. Willard, \emph{General Topology}, Dover Publications, 2004.
\end{thebibliography}
\end{document}